\documentclass[12pt]{amsart}

\usepackage{amsmath}
\usepackage{amssymb}
\usepackage{amsfonts}
\usepackage{amsthm}
\usepackage{enumerate}
\usepackage{hyperref}
\usepackage{color}

\theoremstyle{plain}
\newtheorem{thm}{Theorem}[section]

\newtheorem{prop}[thm]{Proposition}
\newtheorem{lem}[thm]{Lemma}
\newtheorem{cor}[thm]{Corollary}

\theoremstyle{definition}

\newtheorem{exmp}[thm]{Example}

\newtheorem{rem}[thm]{Remark}

\newtheorem{dfns-rems}[thm]{Definitions and Remarks}
\newtheorem{notas-rems}[thm]{Notations and Remarks}
\newtheorem{exmps-rems}[thm]{Examples and Remarks}

\begin{document}
\bibliographystyle{amsplain}


\author[Amir Mousivand]{Amir Mousivand}

\address{A. Mousivand\\Department of Mathematics, Firoozkooh Branch, Islamic Azad
University (IAU), Firoozkooh, Iran.}
\email{amirmousivand@gmail.com\\amir.mousivand@iaufb.ac.ir}

\author[Seyed Amin Seyed Fakhari]{Seyed Amin Seyed Fakhari}

\address{S. A. Seyed Fakhari\\School of Mathematics, Institute for Research
in Fundamental Sciences (IPM), P.O. Box 19395-5746, Tehran, Iran.}
\email{fakhari@ipm.ir}
\urladdr{http://math.ipm.ac.ir/fakhari/}

\author[Siamak Yassemi]{Siamak Yassemi}

\address{S. Yassemi\\School of Mathematics, Statistics and Computer Science, College of
Science, University of Tehran, Tehran, Iran, and School of Mathematics, Institute
for Research in Fundamental Sciences (IPM), P.O. Box 19395-5746, Tehran, Iran.}
\email{yassemi@ipm.ir}
\urladdr{http://math.ipm.ac.ir/yassemi/}


\keywords{edge ideal, vertex decomposable graph, unmixed graph, sequentially Cohen--Macaulay graph, chordal graph}

\subjclass[2010]{Primary: 13H10, 05E45; Secondary:13F55, 05E40}

\thanks{The Research of A. Mousivand is partially supported by a grant from Firoozkooh Branch, Islamic Azad
University. The research of S. A. Seyed Fakhari and S. Yassemi were in part supported by a grant from IPM (No. 92130422) and (No. 921302140).}

\title{A new construction for Cohen-Macaulay graphs}

\begin{abstract}
Let $G$ be a finite simple graph on a vertex set $V(G)=\{x_{11}, \ldots, x_{n1}\}$. Also let $m_1, \ldots, ,m_n \geq 2$ be integers and $G_1, \ldots, G_n$ be connected simple graphs on the vertex sets $V(G_i)=\{x_{i1}, \ldots, x_{im_i}\}$. In this paper, we provide necessary and sufficient conditions on $G_1, \ldots, G_n$ for which the graph obtained by attaching $G_i$ to $G$ is unmixed or vertex decomposable. Then we characterize Cohen--Macaulay and sequentially Cohen--Macaulay graphs obtained by attaching the cycle graphs or connected chordal graphs to an arbitrary graphs.
\end{abstract}

\maketitle

\section{Introduction and preliminaries} \label{sec1}

Let $G$ be finite simple (undirected with no loops or multiple edges) graph on the vertex set $V(G)=\{x_1, \ldots, x_n\}$ whose edge set is $E(G)$. By identifying the vertex $x_i$ with the variable $x_i$ in the polynomial ring $R=\mathbb{K}[x_1, \ldots, x_n]$ over the field $\mathbb{K}$, one can associate an ideal to $G$ whose generators are square-free quadratic monomials $x_ix_j$ with $\{x_i,x_j\}\in E(G)$. This ideal is called the {\it edge ideal} of $G$ and will be denoted by $I(G)$. Also the {\it edge ring} of $G$, denoted by $\mathbb{K}[G]$, is defined to be the quotient ring $\mathbb{K}[G]=R/I(G)$. Edge ideals were first introduced by Villarreal \cite{V1}. Fr\"{o}berg in \cite{Fr} showed that Stanley-Reisner ideals with $2$-linear resolutions can be characterized graph-theoretically. Later the edge ideals were studied by many authors in order to examine their algebraic properties in terms of the combinatorial data of graphs, and vice versa. Among the many papers that have studied the properties of edge ideals, see \cite{Fr,HHKO,Ku,VT,VTV,W} and their references. We call a graph ({\it sequentially}) {\it Cohen--Macaulay} if its edge ring is a (sequentially) Cohen--Macaulay ring.

The {\it independence simplicial complex} of a graph $G$ is defined
by $$\Delta_G=\{A\subseteq V(G)\mid A \,\, \mbox{is an independent set
in}\,\, G\}.$$ We recall that $A\subseteq V(G)$ is an {\it independent set}
in $G$ if none of its elements are adjacent. Note that $\Delta_G$ is
precisely the simplicial complex with the Stanley--Reisner ideal $I(G)$. A graph is called {\it unmixed} if all its maximal independent sets have the same cardinality. It is known that any Cohen--Macaulay graph is unmixed (for example, see \cite[Lemma 9.1.10]{HH}). We call a graph $G$ {\it vertex decomposable} (resp. {\it shellable})  if the independence complex $\Delta_G$ is vertex decomposable (resp. {\it shellable}). Vertex decomposability were introduced in the pure case by Billera and Provan \cite{Bp} and extended to non-pure complexes by Bj\"{o}rner and Wachs \cite{BW}. We have the following implications (see \cite[Theorem 11.3]{BW} and \cite[corollary 8.2.19]{HH})
$$\mbox {vertex~decomposable}\Longrightarrow \mbox {shellable}\Longrightarrow \mbox {sequentially~Cohen--Macaulay}$$
and it is known that the above implications are strict.

In the present paper we consider the graph obtained by attaching a connected simple graph to each vertex of a graph $G$. In \cite[Proposition
2.2]{V1} Villarreal proved that the graph obtained from $G$ by adding a whisker to each vertex is
Cohen{Macaulay. Later Dochtermann and Engstr\"{o}m \cite[Theorem 4.4]{DE} proved that such a
graph is unmixed and vertex decomposable. Adding a whisker to each vertex is the
same as saying that attaching the complete graph $K_2$ to each vertex. Recently, Hibi, Higashitani, Kimura and O'Keefe give a generalization of this result by showing that the graph obtained by attaching a complete graph to each vertex of a graph $G$ is unmixed and vertex decomposable \cite[Theorem 1.1]{HHKO} (see also \cite{CN}). We generalize the above results as follows:

Let $G$ be a finite simple graph on a vertex set $V(G)=\{x_{11}, \ldots, x_{n1}\}$. Also let $m_1, \ldots, m_n \geq 2$ be integers and $G_1, \ldots, G_n$ be connected simple graphs on the vertex sets $V(G_i)=\{x_{i1}, \ldots, x_{im_i}\}$. We use $G(G_1, \ldots, G_n)$ to denote the graph obtained by attaching $G_i$ to $G$ on the vertex $x_{i1}$ for all $i=1, \ldots, n$. If $G(G_1, \ldots, G_n)$ is vertex decomposable, then $G_i$ is vertex decomposable for every $i=1, \ldots, n$. Furthermore the converse is true if $x_{i1}$ is a shedding vertex of the graph $G_i$, for every $i=1, \ldots, n$ (see Proposition \ref{thm:VD}).

In Section \ref{sec3} we study the unmixedness of $G(G_1,...,G_n)$. We show that the graph $G(G_1,...,G_n)$  is unmixed if and only if $G_i$ and $G_i\setminus\{x_{i1}\}$ are unmixed for every $i=1,...,n$ (see Proposition \ref{thm:UN}). Finally we characterize (sequentially) Cohen--Macaulay graphs of the form $G(G_1,...,G_n)$, where every $G_i$ is a cycle graph or a connected chordal graph with at least two vertices (see Theorems \ref{thm:MAIN1} and \ref{thm:MAIN2}).

\section{Vertex decomposability} \label{sec2}

Let $G$ be a simple graph on a vertex set $V(G)=\{x_{11}, \ldots, x_{n1}\}$. Also Let $m_1, \ldots, m_n\geq 2$ be integers and $G_1,\ldots, G_n$ be connected graphs on the vertex sets $V(G_i)=\{x_{i1}, \ldots, x_{im_i}\}$. We use $G(G_1, \ldots, G_n)$ to denote the graph obtained by attaching $G_i$ to $G$ on the vertex $x_{i1}$ for all $i=1, \ldots, n$. In this section we consider vertex decomposability of $G(G_1, \ldots, G_n)$. We first recall the definition of vertex decomposable simplicial complex. It is defined in terms of the deletion and the link of faces of a simplicial complex. For a face $F\in\Delta$, the {\it link} of $F$ is the simplicial complex
$${\rm link}_{\Delta} (F) = \{~G\in\Delta ~|~G\cap F =\emptyset~,~G\cup F \in \Delta~\}$$
while the {\it deletion} of $F$ is the simplicial complex
$${\rm del}_{\Delta} (F) = \{~G\in\Delta ~|~G\cap F =\emptyset~\}.$$
When $F=\{x\}$, we simply write ${\rm link}_{\Delta}(x)$ and ${\rm del}_{\Delta}(x)$. Also we usually use $\Delta\setminus \{x\}$ for ${\rm del}_{\Delta} (x)$.

A simplicial complex $\Delta$ is recursively defined to be {\it vertex decomposable} if it has only one facet (i.e. simplex), or else has some vertex $x$ such that

\begin{itemize}
\item [(1)] Both ${\rm link}_{\Delta}(x)$ and $\Delta\setminus \{x\}$ are vertex decomposable, and
\item [(2)] There is no face of ${\rm link}_{\Delta} \{x\}$ which is also a facet of $\Delta\setminus \{x\}$.
\end{itemize}

A {\it shedding vertex} is the vertex $x$  which satisfies the above conditions.

\begin{rem}

Our definition of shedding vertex is slightly different with the definition in \cite{W}, where a shedding vertex is the one which satisfies only condition (2).

\end{rem}

Let $G$ be a graph. For $W\subseteq V(G)$ , we denote by $G\setminus W$, the induced subgraph of $G$ on $V(G)\setminus W$. For a vertex $x\in V(G)$, let $N_G(x)$ denotes the neighborhood of $x$ in $G$, i.e., $N_G(x)=\{y\in V(G)~|~\{x,y\}\in E(G)\}$, and let $N_G[x] =\{x\}\cup N_G(x)$.  We call a graph $G$ vertex decomposable if the independence complex $\Delta_G$ is vertex decomposable. Therefore we have the following translation of vertex decomposable for graphs (see \cite [Section 2]{W}):

A graph $G$ is vertex decomposable if it is totally disconnected, or else has some vertex $x$ such that

\begin{itemize}
\item [(1)] Both $G\setminus N_G[x]$ and $G\setminus \{x\}$ are vertex decomposable, and
\item [(2)] For every independent set $S$ in $G\setminus N_G[x]$, there exists some $y\in N_G(x)$ such that $S\cup \{y\}$ is independent in $G\setminus \{x\}$.
\end{itemize}

A vertex $x$ satisfying the above conditions is called a shedding vertex for $G$.\\

The following result shows that a graph $G$ is vertex decomposable if and only if, each connected component of $G$ is vertex decomposable.

\begin{lem}
\cite[Lemma 20]{W} Let $G_1$ and $G_2$ be two graphs such that $V(G_1)\cap V(G_2)=\emptyset$, and set $G=G_1\cup G_2$.
Then $G$ is vertex decomposable if and only if $G_1$ and $G_2$ are vertex decomposable.
\end{lem}

We are now ready to state and prove the first main result of this paper.

\begin{prop} \label{thm:VD}
Let $G$ be a simple graph on a vertex set $V(G)=\{x_{11}, \ldots, x_{n1}\}$.
Also Let $m_1, \ldots, m_n\geq 2$ be integers and $G_1, \ldots, G_n$ be connected graphs on the vertex sets $V(G_i)=\{x_{i1}, \ldots, x_{im_i}\}$. Then

\begin{itemize}
\item[(i)] If $G_1, \ldots, G_n$ are vertex decomposable and $x_{11}, \ldots, x_{n1}$ are shedding vertices of $G_1, \ldots, G_n$, respectively, then $G(G_1, \ldots, G_n)$ is vertex decomposable.
\item[(ii)] Conversely, if $G(G_1, \ldots, G_n)$ is vertex decomposable, then $G_1, \ldots, G_n$ are vertex decomposable.
\end{itemize}

\end{prop}

\begin{proof}
(i) We use induction on $n$. If $n=1$, then $G(G_1)=G_1$ and there is nothing to prove. Assume that $n>1$ and the assertion holds for any graph $G$ with at most $n-1$ vertices. We claim that $x_{11}$ is a shedding vertex of $G'=G(G_1, \ldots, G_n)$. To prove this claim, assume that $A$ is an independent set in $G'\setminus N_{G'}[x_{11}]$. We write $A$ as the disjoint union $A=B\cup C$, where $B=A\cap (V(G_2)\cup \ldots \cup V(G_n))$ and $C=A\cap V(G_1)$. Since $N_{G_1}[x_{11}]\subseteq N_{G'}[x_{11}]$, it follows that $C$ is an independent set in $G_1\setminus N_{G_1}[x_{11}]$, which together with the fact that $x_{11}$ is a shedding vertex of $G_1$ implies that there exists $x_{1j}\in N_{G_1}(x_{11})$ such that $C\cup\{x_{1j}\}$ is independent in $G_1\setminus \{x_{11}\}$. Note that for any $z\in V(G_2)\cup...\cup V(G_n)$ one has $\{x_{1j},z\}\notin E(G')$. Hence $A\cup\{x_{1j}\}$ is independent in $G'\setminus \{x_{11}\}$. Next we show that $G'\setminus \{x_{11}\}$ and $G'\setminus N_{G'}[x_{11}]$ are vertex decomposable. We write $G'\setminus \{x_{11}\}=G'_1\cup G'_2$, where $G'_1=G_1\setminus \{x_{11}\}$ and $G'_2=G\setminus\{x_{11}\}(G_2, \ldots, G_n)$. Then $G'_2$ is vertex decomposable by induction hypothesis and $G'_1$ is vertex decomposable since $x_{11}$ is a shedding vertex for $G_1$ and $G_1$ is vertex decomposable. Therefore $G'\setminus \{x_{11}\}$ is vertex decomposable, since its connected components are vertex decomposable. Now consider $G'\setminus N_{G'}[x_{11}]$. If $N_G(x_{11})=\emptyset$,  Then $x_{11}$ is isolated in $G$ and hence $G'$ is the disjoint union of $G_1$ and $(G\setminus \{x_{11}\})(G_2, \ldots, G_n)$. It is obvious that $G'$ is vertex decomposable, since by induction hypothesis its connected components are vertex decomposable. So suppose $N_G(x_{11})\neq \emptyset$ and by relabeling the vertices of $G$, assume that $N_G(x_{11})=\{x_{21}, \ldots, x_{t1}\}$ where $t\leq n$. It is easy to see that $G'\setminus N_{G'}[x_{11}]$ is the disjoint union of the following graphs:

\begin{itemize}
\item[(i)] $G_1\setminus N_{G_1}[x_{11}]$;
\item[(ii)] $G_i\setminus \{x_{i1}\}$ for $i=2, \ldots, t$;
\item[(iii)] $(G\setminus \{x_{11}, \ldots, x_{t1}\})(G_{t+1}, \ldots, G_n)$.
\end{itemize}

Clearly, the graphs in (i) and (ii) are vertex decomposable, since $x_{i1}$ is shedding vertex for $G_i$, and $G_i$ is vertex decomposable for every $i=1, \ldots, n$. Moreover,
by the induction hypothesis, the graph in (iii) is also vertex decomposable. Thus $G'\setminus N_{G'}[x_{11}]$ is vertex decomposable, as required.

(ii) By symmetry it is enough to show that $G_1$ is vertex decomposable. For every $j\geq 2$, let $F_j$ be a facet of $\Delta_{G_j}$ with $x_{j_1}\notin F_j$. It follows that $F=\cup_{j=2}^{n}F_j$ is a face of $\Delta_{G'}$. Since $\Delta_{G'}$ is vertex decomposable, $\rm{link}_{\Delta_{G'}}(F)$ is vertex decomposable. On the other hand, one has $\rm{link}_{\Delta_{G'}}(F)=\Delta_{G_1}$. Hence $\Delta_{G_1}$ is vertex decomposable, i.e., $G_1$ is vertex decomposable.
\end{proof}

\begin{rem} \label{link}
A similar argument as in the proof of the second part of Proposition \ref{thm:VD} shows that if $G(G_1, \ldots, G_n)$ is shellable, sequentially Cohen-Macaulay or Cohen--Macaulay, then $G_1, \ldots, G_n$ are shellable, sequentially Cohen--Macaulay or Cohen--Macaulay, respectively. But the converse is not (in general) true, as the following example shows.
\end{rem}

\begin{exmp}
If $G_1, \ldots, G_n$ are vertex decomposable and we attach each $G_i$ to $G$ in a non-shedding vertex, then $G(G_1, \ldots, G_n)$ may not (in general) be vertex decomposable. For example, the graph $G'$ obtained by attaching $P_2$, the path of length $2$, in a vertex of degree $1$ to every vertex of the cycle of length $4$ (Figure 1) is not vertex decomposable. Because the set of the vertices of degree one of $G'$ is a face of $\Delta_{G'}$ whose link is the independence complex of the cycle of length $4$ which is not vertex decomposable.

\unitlength 1mm 
\linethickness{0.8pt}
\ifx\plotpoint\undefined\newsavebox{\plotpoint}\fi 
\begin{picture}(30,60)(35,-20)

\put(81,20){\line(1,0){10}}

\put(81,10){\line(1,0){10}}

\put(81,20){\line(0,-1){10}}

\put(91,20){\line(0,-1){10}}

\put(81,20){\circle*{1}}
\put(81,10){\circle*{1}}
\put(91,19.9){\circle*{1}}
\put(91,10){\circle*{1}}

\put(71,20){\line(1,0){10}}

\put(71,30){\line(0,-1){10}}

\put(71,30){\circle*{1}}
\put(71,19.9){\circle*{1}}
\put(81,20){\circle*{1}}

\put(91,20){\line(1,0){10}}

\put(101,30){\line(0,-1){10}}

\put(91,20){\circle*{1}}
\put(101,19.9){\circle*{1}}
\put(101,30){\circle*{1}}

\put(71,10){\line(1,0){10}}

\put(71,10){\line(0,-1){10}}

\put(71,10){\circle*{1}}
\put(71,0){\circle*{1}}
\put(81,9.9){\circle*{1}}

\put(91,10){\line(1,0){10}}

\put(101,10){\line(0,-1){10}}

\put(91,10){\circle*{1}}
\put(101,9.9){\circle*{1}}
\put(101,0){\circle*{1}}

\put(80,-10){\bf{Figure 1} }
\end{picture}

Note that the above graph is neither shellable nor sequentially Cohen--Macaulay. In addition, the graph obtained by attaching $P_3$, the path of length $3$, in a vertex of degree $1$ to every vertex of the cycle of length $4$ is not Cohen--Macaulay.
\end{exmp}

A graph is called {\it chordal} if every cycle of
length at least four has a chord. We recall that a {\it chord} of a cycle
is an edge which joins two vertices of the cycle but is not itself an edge
of the cycle. By Dirac's theorem \cite{D} every chordal graph has a {\it simplicial} vertex, i.e., a vertex whose
neighbors form a clique. Woodreefe \cite[Corollary 7]{W} proved that a chordal graph is vertex decomposable and every neighbor of a simplicial vertex is a shedding vertex. Let $C_n$ denotes the cycle of length $n$. Francisco and Van Tuyl in \cite[Theorem 4.1]{FVT} showed that $C_n$ is vertex decomposable (shellable or sequentially Cohen--Macaulay) if and only if $n\in\{3,5\}$. Combining these facts with Proposition \ref{thm:VD}, we conclude the following corollaries.

\begin{cor} \label{cor:KCPVD}
Let $G$ be a finite simple graph. Then the graph $G'$ obtained by attaching one of the following graphs to each vertex of $G$ is vertex decomposable (and so shellable and sequentially Cohen--Macaulay):
\begin{itemize}
\item[(1)] a chordal graph with at least two vertices attached to $G$ in a neighbor of a simplicial vertex;
\item[(2)] $C_5$, the cycle of length 5.
\end{itemize}
\end{cor}

\begin{cor} \label{cor:CVD}
Let $G$ be a simple graph and Let $m_1, \ldots, m_n\geq 3$ be integers. Then $G(C_{m_1}, \ldots, C_{m_n})$ is vertex decomposable (shellable or sequentially Cohen--Macaulay) if and only if $m_i\in\{3,5\}$ for all $i=1, \ldots, n$.
\end{cor}

\begin{rem}
Woodroofe \cite [Theorem 1]{W} showed that if $G$ is a graph with no chordless cycles of length other than $3$ or $5$, then $G$ is vertex decomposable. For every integer $t\geq 1$ one can construct a vertex decomposable graph which contains $t$ chordless cycles of length other than $3$ or $5$; it is enough to use Proposition \ref{thm:VD} and choose $G$ such that it contains $t$ number of such cycles.
\end{rem}

\section{Unmixedness} \label{sec3}

In this section we investigate the unmixedness of $G(G_1, \ldots, G_n)$. The following example shows that $G(G_1, \ldots, G_n)$ is not, in general, unmixed even if $G$ and $G_1, \ldots, G_n$ are unmixed graphs.

\begin{exmp}

The graph $H=C_4(C_4,C_4,C_4,C_4)$ is not unmixed. To see this, it is enough to observe that, by labeling the vertices as in Figure 2, the following sets are  maximal independent sets for $H$:
\begin{itemize}
\item[(1)] $\{x_{11},x_{13},x_{31},x_{33},x_{23},x_{43}\}$,
\item[(2)] $\{x_{11},x_{13},x_{31},x_{33},x_{22},x_{24},x_{42},x_{44}\}$.
\end{itemize}

\unitlength 1mm 
\linethickness{0.8pt}
\ifx\plotpoint\undefined\newsavebox{\plotpoint}\fi 
\begin{picture}(30,60)(35,-20)

\put(81,20){\line(1,0){10}}
\put(81,10){\line(1,0){10}}
\put(81,20){\line(0,-1){10}}
\put(91,20){\line(0,-1){10}}
\put(81,20){\circle*{1}}
\put(81,10){\circle*{1}}
\put(91,19.9){\circle*{1}}
\put(91,10){\circle*{1}}
\put(71,30){\line(1,0){10}}
\put(71,20){\line(1,0){10}}
\put(71,30){\line(0,-1){10}}
\put(81,30){\line(0,-1){10}}
\put(71,30){\circle*{1}}
\put(81,30){\circle*{1}}
\put(71,19.9){\circle*{1}}
\put(81,20){\circle*{1}}
\put(91,30){\line(1,0){10}}
\put(91,20){\line(1,0){10}}
\put(91,30){\line(0,-1){10}}
\put(101,30){\line(0,-1){10}}
\put(91,20){\circle*{1}}
\put(91,30){\circle*{1}}
\put(101,19.9){\circle*{1}}
\put(101,30){\circle*{1}}
\put(71,10){\line(1,0){10}}
\put(71,0){\line(1,0){10}}
\put(71,10){\line(0,-1){10}}
\put(81,10){\line(0,-1){10}}
\put(71,10){\circle*{1}}
\put(71,0){\circle*{1}}
\put(81,9.9){\circle*{1}}
\put(81,0){\circle*{1}}
\put(91,10){\line(1,0){10}}
\put(91,0){\line(1,0){10}}
\put(91,10){\line(0,-1){10}}
\put(101,10){\line(0,-1){10}}
\put(91,10){\circle*{1}}
\put(91,0){\circle*{1}}
\put(101,9.9){\circle*{1}}
\put(101,0){\circle*{1}}

\put(92,18){$_{x_{_{11}}}$}
\put(100,18){$_{x_{_{12}}}$}
\put(100,32){$_{x_{_{13}}}$}
\put(90,32){$_{x_{_{14}}}$}

\put(82,18){$_{x_{_{21}}}$}
\put(80,32){$_{x_{_{22}}}$}
\put(70,32){$_{x_{_{23}}}$}
\put(70,18){$_{x_{_{24}}}$}

\put(82,12){$_{x_{_{31}}}$}
\put(70,12){$_{x_{_{32}}}$}
\put(70,-2){$_{x_{_{33}}}$}
\put(80,-2){$_{x_{_{34}}}$}

\put(92,12){$_{x_{_{41}}}$}
\put(90,-2){$_{x_{_{42}}}$}
\put(100,-2){$_{x_{_{43}}}$}
\put(100,12){$_{x_{_{44}}}$}

\put(80,-10){\bf{Figure 2}}
\end{picture}

\end{exmp}

Note that the idea of the above example may be applied to show that $C_n(\overbrace{C_4,C_4,\ldots,C_4}^{n-times})$ is not unmixed for all $n\geq 3$, but we prove the next more general result. Recall that the {\it independence number} of a graph $G$, denoted by $\alpha(G)$, is the greatest integer $c$ such that $G$ has a maximal independent set of cardinality $c$.

\begin{prop} \label{thm:UN}
Let $G$ be a simple graph on a vertex set $V(G)=\{x_{11}, \ldots, x_{n1}\}$ and suppose that $G$ has no isolated vertex. Assume that $G_1, \ldots, G_n$ are connected graphs on the vertex sets $V(G_i)=\{x_{i1}, \ldots, x_{im_i}\}$, such that $m_i\geq 2$, for every $i=1, \ldots, n$. Then $G(G_1, \ldots, G_n)$ is unmixed if and only if the graphs $G_i$ and $G_i\setminus\{x_{i1}\}$ are unmixed, for every $i=1, \ldots, n$.
\end{prop}

\begin{proof}
First assume that $G'=G(G_1, \ldots, G_n)$ is unmixed. By symmetry it is enough to show that $G_1$ and $G_1\setminus\{x_{11}\}$ are unmixed. For every $j\geq 2$, let $F_j$ be a facet of $\Delta_{G_j}$ with $x_{j_1}\notin F_j$. It follows that $F=\cup_{j=2}^{n}F_j$ is a face of $\Delta_{G'}$. Since $\Delta_{G'}$ is pure, $\rm{link}_{\Delta_{G'}}(F)$ is pure. On the other hand, one has $\rm{link}_{\Delta_{G'}}(F)=\Delta_{G_1}$. Hence $\Delta_{G_1}$ is pure, i.e., $G_1$ unmixed.

Next we show that $G_1\setminus\{x_{11}\}$ is unmixed. Suppose on the contrary that $G_1\setminus\{x_{11}\}$ has two maximal independent sets $B$ and $C$ with different cardinalities. Since $G$ has no isolated vertex, we can choose a vertex $x_{i1}\in N_G(x_{11})$, for some integer $2\leq i\leq n$. Let $A$ be a maximal independent set of $G\setminus\{x_{11}\}(G_2, \ldots, G_n)$ which contains $x_{i1}$. One can easily see that $A\cup B$ and $A\cup C$ are maximal independent sets of $G'$ with different cardinalities, which is impossible.

Conversely, assume that for every $i=1, \ldots, n$ the graphs $G_i$ and $G_i\setminus\{x_{i1}\}$ are unmixed. We show that the cardinality of every maximal independent set of $G'$ is equal to $\alpha(G_1)+ \ldots+ \alpha(G_n)$ an this proves that $G'$ is unmixed. We note that since $x_{i1}$ is not an isolated vertex of $G_i$, for every $i=1, \ldots, n$, the graph $G_i$ has a maximal independent set which does not contain $x_{i1}$. Since $G_i$ is unmixed, we conclude that $\alpha(G_i\setminus\{x_{i1}\})=\alpha(G_i)$, for every $i=1, \ldots, n$. Let $A$ be a maximal independent set of $G'$. If $x_{i1}\in A$, for some $1\leq i\leq n$, then $A\cap V(G_i)$ is a maximal independent set of $G_i$ and so its cardinality is equal to $\alpha(G_i)$. Thus the cardinality of $A\cap V(G_i\setminus\{x_{i1}\})$ is equal to $\alpha(G_i)-1$. On the other hand if $x_{i1} \notin A$, for some $1\leq i\leq n$, then again $A\cap V(G_i)$ is a maximal independent set of $G_i$. But in this case we conclude that the cardinality of $A\cap V(G_i\setminus\{x_{i1}\})$ is equal to $\alpha(G_i)$. Note that $$A=\dot{\bigcup}_{i=1}^n(A\cap V(G_i\setminus\{x_{i1}\}))\dot{\cup} (A\cap \{x_{11}, \ldots, x_{n1}\}.$$ Now summing the cardinalities complete the proof.
\end{proof}

\begin{rem}
Let $G$ be a simple graph on the vertex set $V(G)=\{x_{11}, \ldots, x_{n1}\}$ such that $x_{11}, \ldots, x_{t1}$ ($t\leq n$) are all the isolated vertices of $G$. Then Proposition \ref{thm:UN} shows that $G(G_1, \ldots, G_n)$ is unmixed if and only if $G_i$ is unmixed for every $i=1, \ldots, n$ and $G_i\setminus\{x_{i1}\}$ is unmixed for every $i=t+1, \ldots, n$.
\end{rem}

In the following proposition we restrict ourselves to the family of chordal graphs and prove that unmixedness of these graphs is preserved if we delete a suitable vertex.

\begin{prop} \label{prop:CHORD}
Let $G$ be a unmixed chordal graph and $x$ be a simplicial vertex of $G$. Then $G\setminus \{y\}$ is unmixed, for every $y\in N_G(x)$.
\end{prop}
\begin{proof}
Let $A$ be a maximal independent set of $G\setminus \{y\}$. It is enough to show that the cardinality of $A$ is equal to $\alpha(G)$. Suppose that $A$ does not contain any vertex $z\in N_G(y)$. Since $N_G(x)\subseteq N_G(y)$, we conclude that $A\cup \{x\}$ is an independent set of $G\setminus \{y\}$, which is contradiction. Thus we can assume that $A$ contains a vertex $z\in N_G(y)$. Then $A$ is a maximal independent set of $G$ and so $\mid A\mid=\alpha(G)$.
\end{proof}

Let $C_m$ be a cycle of length $m$ and $x$ be a vertex of $C_m$. One can easily check that $C_m\setminus\{x\}$ is unmixed if and only if $m\in \{3,5\}$. Combining this observation with  Propositions \ref{thm:VD} and \ref{thm:UN} we conclude the following results.

\begin{thm} \label{thm:MAIN1}
Let $G$ be a simple graph on a vertex set $V(G)=\{x_{11}, \ldots, x_{n1}\}$ and suppose that $G$ has no isolated vertex. Assume that $C_1, \ldots, C_m$ $(0\leq m \leq n)$ are cycle graphs and $G_{m+1}, \ldots, G_n$ are connected chordal graphs with at least two vertices. Set $G'=G(C_1, \ldots, C_m, G_{m+1}, \ldots, G_n)$ and assume that for every $i=m+1, \ldots, n$ the graph $G_i$ is attached to $G$ in a neighbor of a simplicial vertex. Then the following conditions are equivalent:
\begin{itemize}
\item[(1)] $G'$ is unmixed;
\item[(2)] $G'$ is Cohen--Macaulay;
\item[(3)] $G'$ is unmixed and shellable;
\item[(4)] $G'$ is unmixed and vertex decomposable;
\item[(5)] $C_i\in\{C_3,C_5\}$ for every $i=1, \ldots, m$ and the graph $G_i$ is unmixed for every $i=m+1, \ldots, n$.
\end{itemize}
\end{thm}

\begin{thm} \label{thm:MAIN2}
Let $G$ be a simple graph on a vertex set $V(G)=\{x_{11}, \ldots, x_{n1}\}$ and suppose that $G$ has no isolated vertex. Assume that $C_1, \ldots, C_m$ $(0\leq m \leq n)$ are cycle graphs and $G_{m+1}, \ldots, G_n$ are connected chordal graphs with at least two vertices. Set $G'=G(C_1, \ldots, C_m, G_{m+1}, \ldots, G_n)$ and assume that for every $i=m+1, \ldots, n$ the graph $G_i$ is attached to $G$ in a neighbor of a simplicial vertex. Then the following conditions are equivalent:
\begin{itemize}
\item[(1)] $G'$ is sequentially Cohen--Macaulay;
\item[(2)] $G'$ is shellable;
\item[(3)] $G'$ is vertex decomposable;
\item[(4)] $C_i\in\{C_3,C_5\}$ for every $i=1, \ldots, m$.
\end{itemize}
\end{thm}
\begin{proof}
The implications $(3)\Rightarrow(2)\Rightarrow(1)$ are always true. The implication $(4)\Rightarrow(3)$ follows from Proposition \ref{thm:VD} and the implication $(1)\Rightarrow(4)$ follows from Remark \ref{link} and using this fact that a cycle graph is sequentially Cohen--Macaulay if and only if it is $C_3$ or $C_5$.
\end{proof}

\par 
\end{document}